\documentclass[11pt]{article}
\usepackage{amsmath,amsthm,amscd,amssymb}
\usepackage{latexsym}
\usepackage{color}

\theoremstyle{plain}
\newtheorem{theorem}{Theorem}[section]
\newtheorem*{theoremA}{Theorem 1}
\newtheorem{lemma}[theorem]{Lemma}
\newtheorem{proposition}[theorem]{Proposition}

\theoremstyle{remark}
\newtheorem{remark}{Remark}[section]
\newtheorem{definition}{Definition}[section]
\newtheorem*{definition*}{Definition}
\newtheorem{example}{Example}[section]

\numberwithin{equation}{section}
\numberwithin{theorem}{section}
\numberwithin{remark}{section}

\def\scal#1#2{\langle #1, #2\rangle}
\def\R#1{\mathbb{R}^{#1}}

\DeclareMathOperator{\re}{\mathrm{Re}}

\DeclareMathOperator{\trace}{\mathrm{tr}}

\DeclareMathOperator{\Trace}{\mathrm{Tr}}

\def\F{\mathbb{F}}
\def\Jord{\mathcal{J}}
\def\Nabla{\,D}
\def\cbot{{e^\bot}}

\title{{A Jordan algebra approach to the cubic eiconal equation}}
\author{Vladimir Tkachev\thanks{Department of Mathematics, Link\"oping University, SE-581 83, Link\"oping, Sweden, tkatchev@kth.se}}

\begin{document}

\maketitle

\begin{abstract}
We establish a natural correspondence between (the equivalence classes of) cubic solutions of $\sum_{i,j=1}^k Q^{ij}\,\frac{\partial u}{\partial x_i}\frac{\partial u}{\partial x_j}=9(\sum_{i,j=1}^k Q_{ij}x_ix_j)^2$ and (the isomorphy classes of) cubic Jordan algebras.
\end{abstract}

\section{Introduction}

The Jordan-Von Neumann-Wigner classification \cite{JordanNeumann} implies that the only possible formally real rank 3 Jordan algebras $J$ are:

\bigskip
\begin{tabular}{p{0.6cm}|p{4.3cm}p{0.05cm}|p{0.05cm}p{8.1cm}}
&a Jordan algebra, $J$ & & &the generic norm of a trace free element, $\sqrt{2}N(x)$
\\\hline
&&&&\\
(i) & $\mathrm{Herm}_3(\mathbb{F}_d),\,\, d=1,2,4,8$ & & &$u_d(x)$
    \\
&  &&&\\
(ii) &  $\R{}\oplus \mathrm{Spin}(\R{n+1})$ & && $4x_n^3-3x_n|x|^2$
  \\
&  &&&\\
(iii) &  $\mathbb{F}_1^3=\R{}\oplus \R{}\oplus\R{}$ & &&  $x_2^3-3x_2x_1^2$
  \\
\end{tabular}

\bigskip
\noindent
where $J$ is equipped with the standard inner product $\scal{x}{y}=\frac{1}{3}\Trace (x\bullet y)$ and
\begin{equation*}\label{CartanFormula0}
    \begin{split}
u_d(x) =x_{3d+2}^3&+\frac{3}{2}x_{3d+2}(|z_1|^2+|z_2|^2-2|z_3|^2-2x_{3d+1}^2)\\
&+\frac{3\sqrt{3}}{2}x_{3d+1}(|z_2|^2-|z_1|^2)+{3\sqrt{3}}\re (z_1z_2z_3),
\end{split}
\end{equation*}
is a Cartan isoparametric cubic. Here $z_k=(x_{kd-d+1},\ldots,x_{kd})\in \R{d}=\mathbb{F}_d$ and $\F_d$ stands for a real division algebra of dimension $d$, i.e. $\mathbb{F}_1=\mathbb{R}$ (the reals), $\mathbb{F}_2=\mathbb{C}$ (the complexes), $\mathbb{F}_4=\mathbb{H}$ (the quaternions) and $\mathbb{F}_8=\mathbb{O}$ (the octonions).
According to a recent result of the author \cite{TkCartan}, the cubic forms displayed in the  last column  are exactly the only  cubic polynomial solutions of
\begin{equation}\label{equivariant1}
|\nabla u(x)|^2=9|x|^4,\quad x\in \R{n}
\end{equation}
which yields a natural question: Does there exist any \textit{explicit} correspondence between cubic solutions of (\ref{equivariant1}) and formally real cubic Jordan algebras?

The concept of Jordan algebra, and formally  real Jordan algebra especially, was introduced originally in the 1930s in \cite{Jordan33A}, \cite{JordanNeumann} in attempt to find an adequate algebraic formalism for quantum mechanics. Nowadays, besides the numerous pure algebraic applications (in particular, the connection between Jordan and Lie theory via the Tits-Kantor-Koecher construction of Lie algebras), Jordan algebras appeared to be very useful in harmonic analysis on self-dual homogeneous cones (Vinberg-Koecher theory), operator theory (JB-algebras), differential geometry (symmetric spaces, projective geometry, isoparametric hypersurfaces), integrable hierarchies (Svinopulov's characterization of generalized KdV equation), statistics (Wishart distributions on Hermitian matrices and on Euclidean Jordan algebras).
In the most interesting for us case of cubic Jordan algebras, the Jordan multiplication structure can be recovered directly from the generic norm by the so-called Springer construction. The cubic case is also very distinguished because its very remarkable appearance in different areas such as Scorza and secant varieties, Cremona transformations, non-classical solutions to fully nonlinear PDEs, extremal black holes and supergravity \cite{Chaput02}, \cite{Chaput03}, \cite{Duff1}, \cite{Eting}, \cite{Kru}, \cite{Nadir}, \cite{Pirio1}, \cite{Pirio2}.

On the other hand, the eiconal equation (\ref{equivariant1}) naturally appears in the context of isoparametric hypersurfaces of the Euclidean spheres, i.e. hypersurfaces having only constant principal curvatures \cite{Cecil}. A celebrated result due to H.F.~M\"unzner \cite{Mun1}, \cite{Mun2}, asserts that the cone over an isoparametric
hypersurface is algebraic and its defining polynomial $u(x):\R{n}\to \R{}$ is homogeneous of degree $g = 1, 2, 3, 4$ or $6$ coinciding the number of distinct principal curvatures and is subject to the following conditions:
\begin{equation}\label{Muntzer1}
|\nabla u(x)|^2=g^2|x|^{2g-2}, \qquad \Delta u(x)=\frac{m_2-m_1}{2}\,g^2|x|^{g-2},
\end{equation}
where $m_1$  and $m_2$ are the multiplicities of the maximal and minimal principal curvatures of $M$ respectively (there holds $m_1=m_2$ when $g$ is odd). A classical result due to \'{E}lie~Cartan states that that any isoparametric  hypersurface with $g=3$ distinct curvatures  is obtained as a level set of a \textit{{harmonic}} cubic polynomial solution of (\ref{equivariant1}) and that there are exactly four (congruence classes of) cubic solutions given by (\ref{CartanFormula0}) for $d=1,2,4,8$.

In this paper we answer the above question by characterizing  the generic norm of an arbitrary Jordan algebra of rank 3 in terms of the general eiconal equation
\begin{equation}\label{equivariant2}
\sum_{i,j=1}^k Q^{ij}\,\frac{\partial u}{\partial x_i}\frac{\partial u}{\partial x_j}=9(\sum_{i,j=1}^k Q_{ij}x_ix_j)^2,\quad x\in \R{k},
\end{equation}
where $Q_{ij}$ is an invertible matrix with $Q^{ij}$ being its inverse. It is more natural, however, to work with an invariant version of (\ref{equivariant2}) which we now define.
Let $(W,Q)$ be a nondegenerate quadratic space, i.e.  $W$ be a vector space over a field $\F$  ($\F$ is always assumed to be $\R{}$ or $\mathbb{C}$) with a quadratic form  $Q:W\to \F$  such that the associated symmetric bilinear form
\begin{equation*}
\label{Qpol}
Q(x;y)=\frac{1}{2}(Q(x+y)-Q(x)-Q(y)), \qquad x, y\in W
\end{equation*}
has the kernel $\{0\}$. The bilinear form $Q(x;y)$ induces a natural inner product on $W$. Let $\nabla u(x)$ denote the covariant derivative of a smooth function $u:W\to \F$, i.e. the unique vector field satisfying the dual realtion
\begin{equation}\label{eqGrad}
Q(\nabla u(x);y)=\partial_y u|_{x}
, \qquad \forall y\in W.
\end{equation}

\medskip
\begin{definition}
A triple $(W,Q,u)$ with a cubic form $u:W\to \F$ on a nondegenerate quadratic space $(W,Q)$ satisfying
\begin{equation}\label{equivariant}
Q(\nabla u(x))=9Q(x)^2, \quad x\in W
\end{equation}
is called a \textit{eiconal triple}.
\end{definition}

\medskip

\begin{remark}
Notice that in the local coordinates $x_i=Q(x,e_i)$ associated with a basis  $\{e_i\}_{1\le i\le n}$ of $W$, the eiconal equation (\ref{equivariant}) takes the form (\ref{equivariant2}) with  $Q_{ij}=Q(e_i;e_j)$. In particular, (\ref{equivariant}) corresponds to the Euclidean norm $Q(x)=|x|^2\equiv \sum_{i=1}^{n}x_i^2$ on  $W=\R{n}$.
\end{remark}

\medskip
Recall that a map $O:(W,Q)\to (\tilde W,\tilde Q)$ between the quadratic spaces $(W,Q)$ and $(\tilde W,\tilde Q)$ is called an isometry of the two quadratic spaces if $O$ is invertible and $\tilde Q(Ox)=Q(x)$ for all $x\in W$. It is straightforward to verify that if $(W,Q,u)$ is an eiconal triple and $O:(W,Q)\to (\tilde W,\tilde Q)$ is an isometry then $(\tilde W,\tilde Q, \tilde u)$ with $\tilde u(O x)=u(x)$ is an eiconal triple too. We call such eiconal triples \textit{equivalent}.

\begin{theoremA}\label{Theorem}
 There exists a  correspondence $\mathcal{J}\xrightarrow{\alpha}\mathcal{E}\xrightarrow{\beta}\mathcal{J}$, $\alpha  \beta=\mathrm{id}_\mathcal{E}$, $\beta  \alpha=\mathrm{id}_\mathcal{J}$, between the equivalence classes of eiconal triples ($\mathcal{E}$) and the isomorphic classes of cubic Jordan algebras ($\mathcal{J}$) and defined as follows.
 \begin{itemize}
 \item[$(\alpha)$]
 Any cubic Jordan algebra  $J=J(V,N,e)$ with the generic norm $N$ gives rise to an eiconal triple 
 $
 (e^\bot,\tau, u)
 $ 
 with the underlying vector space 
 $$
 e^\bot\doteq\{x:\Trace x=0\}
 $$ 
 the quadratic form $\tau(x)=\frac{1}{3}\Trace(x^2)|_{e^\bot}$, and the cubic form $u(x)=\sqrt{2}\,N(x)$. Two isomorphic algebras give rise to equivalent eiconals.

 \item[$(\beta)$]
In the converse direction, any eiconal triple $(W,Q,u)$ gives rise to a cubic Jordan algebra 
$
J=J(V,N,e)
$ 
with the underlying vector space $V=\F\times W$, the generic norm
$$
N(\mathbf{x})=x_0^3-\frac{3}{2} x_0Q(x)+\frac{ u(x)}{\sqrt{2}}, 
$$
and the unit $e=(1,0)$, where $\mathbf{x}=(x_0,x)\in V$. Two equivalent eiconals produces isomorphic algebras.

\end{itemize}

\end{theoremA}

We give the basic information on  Jordan algebras in section~\ref{sec:Jordan} and then prove Theorem~\ref{Theorem} in sections~\ref{sec:Aux} and ~\ref{sec:MainIso}.

\section{Basic facts on Jordan algebras}
\label{sec:Jordan}

Here we review some definitions in Jordan algebra theory following the classical books \cite{JacobsonBook}, \cite{McCrbook} and \cite{FKbook}. A vector space\footnote{We  always assume that $V$ is finite dimensional. } $V$ over a field $\F$ with a commutative $\F$-bilinear map $\bullet:V\times V\to V$  is called a Jordan algebra if it satisfies
\begin{equation}\label{Jordan_quadrat}
x^{2}\bullet (x\bullet y) = x\bullet (x^{ 2}\bullet y),
\end{equation}
where $x^{2}\doteq x\bullet x$. Any Jordan algebra is power associative, i.e. the  subalgebra $V(x)$ generated by a single element $x$ is associative for any $x\in V$.
The rank of $V$ is $r\doteq \max \{\dim V(x): x\in V\}$. An element $x$ is said to be \textit{regular} if $\dim V(x)=r$. For a regular $x\in V$, the elements $e,x,\ldots,x^{r-1}$ are linearly independent, so that
$$
x^r=\sigma_1(x)x^{r-1}-\sigma_2(x)x^{r-2}+\ldots+(-1)^{r-1} \sigma_{r}(x)e,
$$
where $\sigma_i(x)\in \F$ are uniquely defined by $x$. The polynomial
$$
m_x(\lambda)=\lambda^r-\sigma_1(x)\lambda^{r-1}+\ldots+(-1)^r \sigma_{r}(x)\in \F[\lambda]
$$
is called the \textit{minimum polynomial} of the regular element $x\in V$ \cite{JacobOsaka}. The coefficients $\sigma_i(x)$   are polynomial functions of $x$ homogeneous of degree $i$ in the following sense. If $\{e_i\}_{1\le i\le n}$ is an arbitrary  basis of $V$ and
$$
x_\xi\doteq \sum_{i=1}^n \xi_ie_i, \qquad \xi\in \F^n,
$$
is a generic element of the algebra $V$ then
$$
x_\xi^r-\sigma_1(\xi)x_\xi^{r-1}+\ldots+(-1)^r \sigma_{r}(\xi)e=0,
$$
where $\sigma(\xi)$ is a polynomial function of $x$ homogeneous of degree $i$.

\begin{definition*}
The coefficient $\sigma_1(x)=\Trace x $ is called the \textit{generic trace} of $x$, and $\sigma_n(x)=N(x)$ is called the \textit{generic norm} of $x$.
\end{definition*}

\begin{example}
A standard example of a Jordan algebra is the vector space $M_n(\F_d)$ of matrices of size $n\times n$ with entries in  $\F_1=\R{}$  or $\F_2=\mathbb{C}$, and the bullet product
\begin{equation}
\label{jordanproduct}
x\bullet y\doteq \frac{1}{2}(xy+yx).
\end{equation}
In this case, $\Trace x=\trace x$, $N(x)=\det x$  are the usual  trace and determinant of $x$. In fact, it is readily verified that the real vector space $\mathrm{Herm}_3(\mathbb{F}_d)$ of self-adjoint matrices over the real division algebra $\F_d$ with $d=1,2,4$ and arbitrary $n\ge1$, and with $d=8$ and $n\le 3$ is also a Jordan algebra with respect to the multiplication (\ref{jordanproduct}).

\end{example}

\begin{example}\label{ex:spin}
The following well-known construction associates a Jordan algebra structure  to an arbitrary  quadratic space $(V,Q)$ over  $\F$ with a basepoint $e$, i.e. $Q(e)=1$, by virtue of the following  multiplication
\begin{equation}\label{xbullety}
x\bullet y=Q(x;e)y+Q(y;e)x-Q(x;y)e.
\end{equation}
It is readily that (\ref{xbullety}) satisfies the Jordan identity (\ref{Jordan_quadrat}) with the unit element $e$. Thus obtained  Jordan algebra $\mathrm{Spin}(V,Q,e)$ is called a \textit{spin factor} (or a quadratic factor). It follows from (\ref{xbullety}) that any element $x\in V$ satisfies the quadratic identity
$$
x^2-2Q(x,e)x+Q(x)e=0
$$
which, in particular, shows that $\Trace (x)=2Q(x,e)$ and $N(x)=Q(x)$. Conversely, any rank 2 Jordan algebra is obtained by virtue of the above construction from an appropriate quadratic space. In any case, the multiplicative structure on a quadratic Jordan algebra is completely determined by its generic norm.
\end{example}

In the case of Jordan algebra of rank 3 one  still has a possibility to restore the multiplicative structure from the generic norm, but in contrast to the quadratic case only very special cubic forms can be used to construct cubic Jordan algebras. The corresponding procedure is called the Springer construction \cite{Springer59} and we recall below following  \cite[II.4.2]{McCrbook}. Let  $V$ be  a vector space over a field $\F$ with a cubic form $u:V\to V$. Notice that the set  $u^{-1}(1)=\{x\in V: u(e)=1\}$ is nonempty because  $u$ is an odd function. A vector $e\in u^{-1}(1)$ is called a basepoint. Consider the decomposition
\begin{equation}\label{equation}
u(x+ty)=u(x)+tu(x;y)+t^2u(y;x)+t^3u(y), \qquad x,y\in V,\,\,t\in \F,
\end{equation}
where $u(x;y)=\partial_y u|_{x}$
is the directional derivative of $u$ in the direction $y$ evaluated at $x$.
Since  $u(x;y)$  is a quadratic form with respect to the first variable, one can polarize it to obtain
the full linearization of the cubic form $u$:
\begin{equation}\label{full}
u(x;y;z)=u(x+y;z)-u(x;z)-u(y;z)= \partial_x\partial_y\partial_z u|_{e}.
\end{equation}
Then Euler's homogeneity function theorem implies
\begin{equation}\label{u40}
u(x)=\frac{1}{3}u(x;x)=\frac{1}{6}u(x;x;x).
\end{equation}

Now we fix some cubic form $N(x)\not\equiv 0$ on $V$ and its basepoint $e\in N^{-1}(1)$. Linearizing $N$ as above, one defines the \textit{linear trace} form and and the \textit{quadratic spur} form to be
\begin{equation}\label{traceform}
\Trace(x)=N(e;x)=N(e;e;x)
\end{equation}
and
\begin{equation}\label{SPUR}
S(x)=N(x;e)=N(x;x;e),
\end{equation}
respectively. Then
\begin{equation}\label{TC3}
\Trace(e)=S(e)=3N(e)=3.
\end{equation}
 Polarizing the quadratic spur form,
$$
S(x;y)=N(x;y;e)\equiv S(x+y)-S(x)-S(y)
$$
we define the bilinear \textit{trace form} by
\begin{equation}\label{tracec}
T(x;y)=\Trace(x)\Trace(y)-S(x;y)\equiv -\partial_x\partial_y \log N_{ce}
\end{equation}

\begin{definition}\label{def:1}
Let $V$ be a finite-dimensional vector space over $\F$ and let $N:V\to \R{}$ be a  cubic form with basepoint $e$. $N$ is said to be a \textit{Jordan cubic form},  if the following is satisfied:
\begin{itemize}
  \item[$(\mathrm{i})$]
  the trace bilinear form $T(x;y)$ is a nondegenerate bilinear form;
  \item[$(\mathrm{ii})$]
  the quadratic sharp map $\#:V\to V$, defined (uniquely) by
  \begin{equation}\label{sharpdef}
  T(x^\#;y)=N(x;y),
  \end{equation}
    satisfies the adjoint identity
  \begin{equation}\label{adjoint}
(x^{\#})^\#=N(x)x.
\end{equation}
\end{itemize}
\end{definition}

\begin{example}
We demonstrate the above definitions by the Jordan algebra of symmetric $3\times 3$ matrices  $V=\mathrm{Herm}_3(\R{})$ with the unit matrix being a basepoint: $e=\mathbf{1}$ and a cubic form $N(x)=\det x$. Then
$$
\Trace x=\partial_x N|_{e}=(\det(\mathbf{1}+tx))'_{t=0} =\trace x,
$$
and, by virtue of the Newton identities,
$$
S(x)=\partial_e N|_{x}=(\det(x+t\mathbf{1}))'_{t=0}=\frac{1}{2}((\trace x)^2-\trace x^2).
$$
Polarizing the latter identity  yields $S(x;y)=\trace x\trace y-\trace x\bullet y$. Thus,
$$
T(x;y)=\trace x\bullet y\equiv \trace xy.
$$
Now observe that $T(x;y)=\trace xy$ obviously is a nondegenerate bilinear form, hence we get (i) in Definition~\ref{def:1}.
On the other hand, if $x$ is an invertible matrix then
$$
3N(x;y)=\partial N_y|_{x}=(\det(x+ty))'_{t=0}=(\det x\,\det(\mathbf{1}+tyx^{-1}))'_{t=0}=\det x\,\trace( y x^{-1}),
$$
which yields by (\ref{sharpdef}) the adjoint map
$x^\#=x^{-1}\det x$
is the adjoint matrix in the usual sense of linear algebra, thus $x^{\#\#}=x\det x$ and the adjoint identity (\ref{adjoint}) follows.
\end{example}

\begin{proposition}[The Springer construction, \cite{McCrbook}, p.~77]\label{pro:springer}
To any Jordan cubic form $N:V\to \F$ with basepoint $e\in V$ one can associate a unital Jordan algebra $\Jord(V,N,e)$ of rank $3$ with the unit $e$ and the Jordan multiplication determined by the formula
\begin{equation}\label{XYtimes}
x\bullet y=\frac{1}{2}(x\# y+\Trace(x)y+\Trace(y)x-S(x;y)e),
\end{equation}
where $x\# y=(x+y)^\#-x^\#-y^\#$. All elements of $\Jord(V,N,e)$ satisfy the cubic identity
\begin{equation}\label{thecubic}
x^{3}-\Trace(x)x^{2}+S(x)x-N(x)e=0,
\end{equation}
and the sharp identity
\begin{equation}\label{sharpidentity}
x^\#=x^{ 2}-\Trace(x)x+S(x)e.
\end{equation}
\end{proposition}

We mention some useful consequences of the above construction. By (\ref{sharpdef})
$ T(x^\#;e)=N(x;e)=S(x)$. Polarizing  this identity one obtains
$
T(x\# y;e)=S(x;y).
$
On the other hand, (\ref{tracec}) and (\ref{TC3}) yield
\begin{equation*}\label{TraceandT}
T(x;e)=\Trace(x)-S(x;e)=\Trace(x)-N(x;e;e)=\Trace(x),
\end{equation*}
hence
\begin{equation}\label{seting}
\Trace(x\# y)=T(x\# y;e)=S(x;y).
\end{equation}
Applying the trace to (\ref{sharpidentity}) and using the last identity
one readily finds  that
\begin{equation}\label{bullet}
\begin{split}
S(x;y)=\Trace(x)\Trace(y)-\Trace(x\bullet y),
\end{split}
\end{equation}
which in its turn yields
\begin{equation}\label{scalproduct}
\Trace(x\bullet y)=T(x;y).
\end{equation}
Setting $x=y$ in (\ref{seting}) yields
\begin{equation}\label{spuridentity}
\Trace(x^\#)=S(x).
\end{equation}

\section{An auxiliary lemma}\label{sec:Aux}

Here we establish some auxiliary formulas for the generic norm which also have an independent interest. Let $N:V\to \F$  be a Jordan cubic form with basepoint $e$ and let $J=J(V,N,e)$ be the associated  cubic Jordan algebra. Notice that the bilinear trace form $T(x;y)$, and thus the inner product
\begin{equation}\label{tau}
\tau(x;y)\doteq \frac{1}{3}T(x;y)
\end{equation}
and the associated quadratic norm $\tau(x)=\tau(x,x)$ are all nondegenerate on $V$. The normalized factor is chosen to ensure that the basepoint $e$ becomes  a unit vector in the induced inner structure. Let us denote by $\Nabla$ the covariant derivative on $(V,T)$. Then $\partial_y N|_{x}=\tau(\Nabla N(x);y)$ for any $y\in V$, hence the sharp map definition (\ref{sharpdef})  yields
$$
\tau(x^\#;y)=\frac{1}{3}N(x;y)=\frac{1}{3}\partial_y N|_{x}=\frac{1}{3}\tau(\Nabla N(x);y),
$$
therefore by the nondegenerateity of $T$
\begin{equation}\label{nablaN}
\Nabla N(x)=3x^\#.
\end{equation}
Let us also denote
\begin{equation}\label{tdef}
t\doteq \tau(x;e)=\frac{1}{3}\Trace (x).
\end{equation}
It is natural to think of $t$ as a `time' coordinate function on $V$. Then the orthogonal complement to the time-line $\F e$ coincides with the subspace of trace free elements of $J$,
$$
\cbot\doteq \{x\in J:\, \tau(x;e)=0\}= \{x\in J:\, \Trace (x)=0\},
$$
which can also be thought as the spatial subspace. This yields  the orthogonal decomposition $V=\F e\oplus \cbot$ and the corresponding decomposition of the covariant derivative into the time- and the spatial components,
\begin{equation}\label{nablaorthogonal}
\Nabla= \partial_t\oplus\nabla,
\end{equation}
where $\nabla$ denotes the covariant derivative induced by $\Nabla$ on $\cbot$. 


\begin{lemma}\label{pro:gen}
In the made notation,
\begin{equation}\label{u0}
\begin{split}
\partial_t N(x)&=S(x),
\end{split}
\end{equation}
\begin{equation}\label{nablaN3}
\tau(\Nabla N(x))=3(\partial_t N(x))^2-18tN(x),
\end{equation}
\begin{equation}\label{nablaInf}
\tau(\Nabla N(x);\Nabla \tau(\Nabla N(x))=108\,\tau(x)N(x).
\end{equation}

\end{lemma}

\begin{proof}
The first relation follows from the definition of $t$ and (\ref{spuridentity}):
\begin{equation}\label{u1}
\begin{split}
\partial_t N(x)&\equiv \tau(\Nabla N(x);e)=\Trace(x^\#)=S(x).
\end{split}
\end{equation}
Next, applying (\ref{tracec}), (\ref{spuridentity}) and  (\ref{adjoint}), we find
\begin{equation}\label{u2}
\begin{split}
\tau(\Nabla N(x))&\equiv 3T(x^\#,x^\#)\\
&=3\Trace(x^\#)^2-3S(x^\#;x^\#)\\
&=3\Trace(x^\#)^2-6S(x^\#)\\
&=3S(x)^2-6\Trace(x^{\#\#})\\
&=3S(x)^2-18N(x)t,
\end{split}
\end{equation}
hence (\ref{u0}) yields
\begin{equation}\label{rew}
\tau(\Nabla N(x))=3(\partial_t N(x))^2-18N(x)t,
\end{equation}
thus  (\ref{nablaN3}) follows.

Finally, in order to show (\ref{nablaInf}), we rewrite (\ref{rew}) by virtue (\ref{u0}) and (\ref{tdef}) as
$$
\tau(\Nabla N(x))=3S(x)^2-6N(x)\Trace(x),
$$
which yields by virtue of (\ref{spuridentity})
\begin{equation}\label{u4}
\begin{split}
\frac{1}{3}\tau(\Nabla N(x);\Nabla \tau(\nabla N(x))&\equiv \partial_{x^{\#}}\tau(\nabla N(x))\\
&=6S(x;x^\#)S(x)-6N(x)\Trace(x^\#)-6N(x;x^\#)\Trace(x)\\
&=6(S(x;x^\#)-N(x))\Trace x^\#-6T(x^\#;x^\#)\Trace(x)\\
\end{split}
\end{equation}
By (\ref{sharpidentity}), $x\bullet x^\#=N(x)e$, hence using (\ref{bullet}), we find
$$
S(x;x^\#)=\Trace(x)\Trace(x^\#)-\Trace(x\bullet x^\#)=\Trace(x)\Trace(x^\#)-3N(x).
$$
On the other hand, by (\ref{u2})
$$
T(x^\#;x^\#)=S(x)^2-2N(x)\Trace(x)=(\Trace(x^\#))^2-2N(x)\Trace(x),
$$
and substituting the found identities into (\ref{u4}) and using (\ref{tracec}) yields the required relation:
\begin{equation}\label{u5}
\begin{split}
\frac{1}{3}\tau(\Nabla N(x);\Nabla \tau(\nabla N(x))&=12\,(\Trace(x)^2-2\Trace(x^\#))N(x)\\
&=12\,(\Trace(x)^2-S(x,x))N(x)\\
&=12\,T(x,x)N(x)\equiv 36\,\tau(x)N(x).
\end{split}
\end{equation}
\end{proof}

\begin{remark}
We demonstrate the above equations for a formally real Jordan algebra $J$. Observe that in this case $J$ is Euclidean, i.e. the bilinear form $\tau$ is positive definite, see for instance \cite[p.~46, p.~154]{FKbook}, hence one can choose local coordinates in $V$ such that $\tau(x)=|x|^2$ and $t=x_n$. Since $2S(x)=(\Trace x)^2-\Trace x^2=9x_n^2-3|x|^2$, the equation  (\ref{u0}) and (\ref{nablaN3}) yield respectively
\begin{equation*}
\begin{split}
\partial_{x_n} N(x)&=\frac{3}{2}(3x_n^2-| x|^2),\\
|\Nabla N(x)|^2&=3(\partial_{x_n} N(x))^2-18x_nN(x),\\
\end{split}
\end{equation*}
while (\ref{nablaInf}) shows that $N(x)$ satisfies an eigenfunction type equation
$$
\Delta_{\infty}N(x)=108 |x|^2 N(x)
$$
for the $\infty$-Laplace operator $\Delta_{\infty}u=\frac{1}{2}\nabla u\cdot \nabla|\nabla u|^2$.
\end{remark}

\bigskip
\section{The proof of Theorem~\ref{Theorem}}
\label{sec:MainIso}

The proof of the theorem follows from Propositions~\ref{cor:from}, Proposition~\ref{pro:th:1} and Proposition~\ref{cor:1} below.

\begin{proposition}\label{cor:from}
Let  $J=J(V,N,e)$ be a cubic Jordan algebra, $\cbot=\{x\in V:\Trace x=0\}$, $u(x)=\sqrt{2}N(x)|_{e^\bot}$  and the inner product $\tau$ is defined by $(\ref{tau})$.
Then $(\cbot,\tau, u)$  is an eiconal triple. Moreover, isomorphic Jordan algebras give rise to equivalent eiconal triples.

\end{proposition}

\begin{proof}
Notice that by (\ref{nablaorthogonal}), $\tau(\Nabla N(x))=\tau(\nabla N(x))+(\partial_t N(x))^2$, hence (\ref{nablaN3}) yields
\begin{equation}\label{setting}
 \tau(\nabla N(x))=2(\partial_t N(x))^2|_{t=0}, \qquad x\in e^\bot.
\end{equation}
On the other hand, by (\ref{tracec}), $S(x)=-\frac{1}{2}T(x,x)=-\frac{3}{2}\tau(x)$ for any $x\in e^\bot$, thus $\partial_t N(x)|_{t=0}=-\frac{3}{2}\tau(x)$ by (\ref{u0}) and setting the found relation into (\ref{setting}) yields
\begin{equation}\label{nablaN3sekv}
\tau(\nabla N(x))=\frac{9}{2}\tau(x)^2
\end{equation}
which implies that $(e^\bot, \tau, \sqrt{2}N|_{e\bot})$ is an eiconal triple, as required.

Now we suppose that $J_1$ and $J_2$ are isomorphic Jordan algebras and denote by $\phi:J_1\to J_2$ the corresponding isomorphism. Observe that any element $\xi\in J_2$  satisfies  the characteristic relation
$$
\xi^{3}-\Trace_2(\xi)\xi^{2}+S_2(\xi)\xi-N_2(\xi)e_2=0,
$$
where $\Trace_k$, $S_k$, $N_k$ and $e_k$ denote the generic trace, the generic spur, the generic norm and the unit element in the algebra  $J_k$, $k=1,2$. Applying $\phi^{-1}$ to the latter relation we find by isomorphy of $\phi$ that
$$
(\phi^{-1}(\xi))^{3}-\Trace_2(\xi)(\phi^{-1}(\xi))^{2}+S_2(\xi)\phi^{-1}(\xi)-N_2(\xi)e_1=0.
$$
Setting $x=\phi^{-1}(\xi)$ in the latter relation, we find that any element $x\in J_1$ satisfies the cubic relation
$$
x^{3}-\Trace_2(\phi(x))x^{2}+S_2(\phi(x))x-N_2(\phi(x))e_1=0.
$$
Since the minimum polynomial is unique, see, for instance \cite{JacobOsaka}, the latter relation implies $\Trace_1(x)=\Trace_2(\phi(x))$, $S_1(x)=S_2(\phi(x))$ and
\begin{equation}\label{NN}
N_1(x)=N_2(\phi(x))
\end{equation}
for any $x\in J_1$. Using (\ref{tracec}) one finds $T_1(x;y)=T_1(\phi(x);\phi(y)),$
which  yields $\tau_1(x;y)=\tau_1(\phi(x);\phi(y))$. Since $\phi(e_1)=e_2$ and $\phi$ is an isomorphism on the level of vector space, one has $\phi(e_1^\bot)=e_2^\bot$, hence $\phi: e_1^\bot\to e_2^\bot$ is  an isometry of the quadratic vector spaces $(e_1^\bot,\tau_1)$ and $(e_2^\bot,\tau_2)$, and, moreover, (\ref{NN}) implies
$$
N_1|_{e_1^\bot}(x)=N_2|_{e_2^\bot}(\phi(x)),
$$
thus  $(e_1^\bot,\tau_1, N_1|_{e_1^\bot})$ and $(e_2^\bot,\tau_2, N_2|_{e_2^\bot})$ are equivalent, which finishes the proof.

\end{proof}

\begin{proposition}\label{pro:th:1}
Let $(W,Q, u)$ be an eiconal triple. Then the cubic form
\begin{equation}\label{Norma}
N(\mathbf{x})\equiv N_u(\mathbf{x})\doteq x_0^3-\frac{3x_0Q(x)}{2} +\frac{u(x)}{\sqrt{2}}, \quad \mathbf{x}=(x_0,x)\in V,
\end{equation}
on the vector space $V=\F\times W$ is Jordan with respect to the basepoint $\mathbf{e}=(1,0)$. In particular, the associated trace bilinear form is
\begin{equation}\label{TT}
T(\mathbf{x};\mathbf{y})=3x_0y_0+3Q(x;y)
\end{equation}
and the Jordan multiplication structure on $J(V,N,c)$ is determined explicitly by
\begin{equation}\label{multi01}
\mathbf{x}^2=(x_0^2+Q(x), \,2x_0x+\frac{1}{3\sqrt{2}}\nabla u(x)), \quad \mathbf{x}=(x_0,x).
\end{equation}
Two equivalent eiconal triples produce isomorphic Jordan algebras.

\end{proposition}

\begin{proof}
Our first step is to establish an auxiliary relation (\ref{HHy2}) below. To this end, we recall that
\begin{equation}\label{QQ}
Q(\nabla u(x);y)=\partial_y u|_{x}\equiv u(x;y) \quad \forall y\in W,
\end{equation}
and notice that the covariant gradient of $u$ is a quadratic endomorphism of $W$ and its normalized polarization
\begin{equation}\label{hessdef}
h(x;y)=\frac{1}{2}(\nabla u(x+y)-\nabla u(x)-\nabla u(y))
\end{equation}
is obviously a symmetric $\F$-bilinear map. It follows from (\ref{hessdef}) and the homogeneity of $\nabla u(x)$  that
\begin{equation}\label{hnabla}
h(x)\doteq  h(x;x)=\nabla u(x).
\end{equation}
Combining (\ref{hessdef}) and (\ref{QQ}) with (\ref{full}) we find
$$
2Q(h(x;y);\,z)=Q(\nabla u(x+y)-\nabla u(x)-\nabla u(y);\,z)=u(x+y;z)-u(x;z)-u(y;z)=u(x;y;z)
$$
for all $x,y,z\in V$ which implies
\begin{equation}\label{HessianX2}
Q(x;\, h(y;z))=Q(h(x;y)\, z).
\end{equation}
Now rewrite (\ref{equivariant}) by virtue of (\ref{hnabla}) as
\begin{equation}\label{Ahx}
Q(h(x))=9Q^2(x),
\end{equation}
hence applying the directional derivative in the direction $y$ to the latter equation yields
\begin{equation}\label{Axx}
Q(h(x);\,h(x;y))=9Q(x)\,Q(y;x)\equiv 9Q(y;\,Q(x)x),
\end{equation}
where $Q(x,y)$ is defined by (\ref{Qpol}).
Using  (\ref{HessianX2}) we find for the left hand side of (\ref{Axx}) that $Q(h(x);\,h(x;y)))=Q(y;\,h(x;h(x))),$
therefore by the non-degeneracy of $Q$,
\begin{equation}\label{HHy}
h(x;h(x))=9Q(x)x.
\end{equation}
Polarizing the latter identity, we obtain
\begin{equation*}
h(y;h(x))+2h(x;h(x;y))=18Q(x;y)x+9Q(x)y,
\end{equation*}
thus,  setting  $y=h(x)$ in the obtained identity yields
\begin{equation}\label{HHy1}
h(h(x))+2h(x;h(x;h(x)))=18Q(x;h(x))x+9Q(x)h(x),
\end{equation}
where we have used $h(h(x);h(x))=h(h(x))$. Moreover, by (\ref{HHy})
$$
h(x;h(x;h(x)))=9Q(x)h(x;x)=9Q(x)h(x),
$$
and by (\ref{hnabla}) and the homogeneity of $u$
\begin{equation}\label{wehave}
Q(x;\,h(x))=Q(x;\nabla u(x))=\partial_x u|_{x}=3u(x),
\end{equation}
therefore (\ref{HHy1}) becomes
\begin{equation}\label{HHy2}
\begin{split}
h(h(x))
&=18Q(x;h(x))x-9Q(x)h(x)
=54u(x)x-9Q(x)h(x).
\end{split}
\end{equation}

Now we are ready to verify that (\ref{Norma}) is Jordan. Notice that $N(\mathbf{e})=1$, hence $\mathbf{e}$ is a basepoint of $N$. Let $\mathbf{x}=(x_0,x)$ and $\mathbf{y}=(y_0,y)$ are chosen arbitrarily in $V=\F\times W$. Since by (\ref{HessianX2})
$$
Q(x;h(x;y))=Q(y;h(x;x))=Q(y;h(x)),
$$
we find
\begin{equation}\label{direct}
\begin{split}
N(\mathbf{x};\mathbf{y})\equiv \partial_\mathbf{y} N|_\mathbf{x}
&=3\biggl(x_0^2-\frac{Q(x)}{2}\biggr)y_0-3x_0Q(x;y)+\frac{ Q(y;h(x))}{\sqrt{2}}.
\end{split}
\end{equation}
Setting  $\mathbf{x}=\mathbf{e}$ in (\ref{direct}) we find by (\ref{traceform}) that
$\Trace(\mathbf{y})=\partial_\mathbf{y} N|_\mathbf{e}=3y_0,$  and setting $\mathbf{y}=\mathbf{e}$ in (\ref{direct}), we find by (\ref{SPUR}) that
$S(\mathbf{x})=\partial_\mathbf{e }N|_\mathbf{x}=3(x_0^2-\frac{Q(x)}{2})$, hence
\begin{equation*}\label{SS}
S(\mathbf{x};\mathbf{y})=S(\mathbf{x}+\mathbf{y})-S(\mathbf{x})-S(\mathbf{y})=6x_0y_0-3Q(x;y),
\end{equation*}
and therefore  (\ref{tracec}) yields
\begin{equation}\label{Tco}
T(\mathbf{x};\mathbf{y})=\Trace(\mathbf{x})\Trace(\mathbf{y})-S(\mathbf{x};\mathbf{y})=3x_0y_0+3Q(x;y),
\end{equation}
thus proving (\ref{TT}).
In order  to establish the adjoint identity  (\ref{adjoint}) note that $Q$ is nondegenerate by the assumption, hence $T$ is nondegenerate too, which implies that $\mathbf{x}^\#$ is well-defined. To determine the adjoint element explicitly we note that $N(\mathbf{x};\mathbf{y})=T(\mathbf{x}^\#;\mathbf{y})$, hence comparing (\ref{direct}) and (\ref{Tco}) readily yields
\begin{equation}\label{sharp}
\mathbf{x}^\#\equiv(x_0^\#,x^\#)=\biggl(x_0^2-\frac{Q(x)}{2}\,,\,\frac{h(x)}{3\sqrt{2}}-x_0x\biggr)
\end{equation}
Let us  denote $\mathbf{x}^{\#\#}\equiv(x_0^{\#\#},x^{\#\#})$. Then by (\ref{sharp})
\begin{equation}\label{sharp1}
\begin{split}
x_0^{\#\#}&=(x_0^\#)^2-\frac{Q(x^\#)}{2}=\biggl(x_0^2-\frac{Q(x)}{2}\biggr)^2-\frac{Q(x^\#)}{2}
\end{split}
\end{equation}
where by virtue (\ref{wehave}) and (\ref{Ahx}),
\begin{equation*}
\begin{split}
Q(x^\#)&=Q\bigl(\frac{h(x)}{3\sqrt{2}}- x_0x\bigr)\\
&=x_0^2Q(x)-\frac{\sqrt{2}}{3}x_0Q(x;h(x))+\frac{ Q(h(x))}{18}\\
&=x_0^2Q(x)-\sqrt{2} x_0u(x)+\frac{Q(x)^2}{2},
\end{split}
\end{equation*}
and substituting this into  (\ref{sharp1}) yields
\begin{equation}\label{sharp2}
\begin{split}
x_0^{\#\#}&=N(\mathbf{x})x_0.
\end{split}
\end{equation}

Arguing as above one obtains
\begin{equation}\label{sharp3}
\begin{split}
x^{\#\#}&=- x_0^\#x^\#+\frac{1}{3\sqrt{2}}\, h(x^\#)\\
&=(x_0^2-\frac{Q(x)}{2})\bigl( x_0x-\frac{ h(x)}{3\sqrt{2}}\,\bigr)+\frac{1}{3\sqrt{2}}\,h\biggl(\frac{ h(x)}{3\sqrt{2}}-x_0x\biggr),
\end{split}
\end{equation}
where using (\ref{HHy}), (\ref{HHy2}) one has
\begin{equation*}\label{sharp4}
\begin{split}
h\biggl(\frac{ h(x)}{3\sqrt{2}}-x_0x\biggr)&=
\frac{h(h(x))}{18} -\frac{\sqrt{2}\, x_0}{3}h(x;h(x))+x_0^2h(x) \\
&=\frac{6u(x)x-Q(x)h(x)}{2}-3\sqrt{2} x_0Q(x)x+x_0^2h(x) \\
&=3(u(x)-\sqrt{2} x_0 Q(x))x+(x_0^2-\frac{Q(x)}{2})h(x)
\end{split}
\end{equation*}
thus (\ref{sharp3}) implies $x^{\#\#}=N(\mathbf{x})x.$
Combining this with (\ref{sharp2})  yields the required adjoint identity $\mathbf{x}^{\#\#}=N(\mathbf{x})\mathbf{x}$ and proves that $N(\mathbf{x})$ is Jordan.

Next,  (\ref{multi01}) follows immediately from the sharp identity (\ref{sharpidentity})
$$
\mathbf{x}^{ 2}=\mathbf{x}^\#+\Trace(\mathbf{x})\mathbf{x}-S(\mathbf{x})\mathbf{e}
$$
and $\Trace(\mathbf{x})=3x_0$, $S(\mathbf{x})=3x_0^2-\frac{3}{2}Q(x)$.

Finally, in order to finish the proof, we suppose that two eiconal triples $(W,Q,u)$ and $(\tilde W,\tilde Q, \tilde u)$ are equivalent, i.e.  there exists a linear map $O:(W,Q)\to (\tilde W,\tilde Q)$ such that $O$ is invertible, $\tilde Q(Ox)=Q(x)$ for all $x\in W$, and $u(x)=\tilde u(Ox)$.
The map $\psi:J(u)\to J(\tilde u)$ by $\psi(x_0,x)=( x_0,O\tilde x)$ is obviously a isomorphism of $W=\F\times V$ and $\tilde{W}=\F\times \tilde{V}$ on the level of vector spaces. Furthermore,  applying (\ref{multi01}) we obtain
\begin{equation}\label{Qhess}
\begin{split}
\psi(\mathbf{x})^{2}&=(x_0^2+\tilde Q(Ox),\,2x_0\,Ox+\frac{1}{3\sqrt{2}} \nabla \tilde u(Ox)).
\end{split}
\end{equation}
On the other hand,
$$
\partial_y u|_{x}\equiv \lim_{t\to 0}\frac{u(x+ty)-u(x)}{t}=\lim_{t\to 0}\frac{\tilde u(Ox+tOy)-\tilde u(Ox)}{t}\equiv
\partial_{Oy} \tilde u|_{Ox}
$$
which yields by the definition of the covariant gradient
$$
Q(\nabla u(x);y)=\partial_y u|_{x}=\partial_{Oy} \tilde u|_{Ox}=\tilde Q(\nabla \tilde u(Ox);Oy)=
Q(O^{-1}\nabla \tilde u(Ox);y),
$$
hence by nondegeneracy of $Q$, $\nabla u(x)=O^{-1}\nabla \tilde u(Ox)$. Hence  (\ref{Qhess}) yields
\begin{equation*}
\begin{split}
\psi(\mathbf{x})^{2}
&=\bigl(x_0+Q(x),\,O(2x_0x+\frac{1}{3\sqrt{2}} \nabla u(x))\bigr)=\psi(\mathbf{x}^{ 2}).
\end{split}
\end{equation*}
The polarization of the obtained identity implies that $\psi(\mathbf{x}\bullet \mathbf{y})=\psi(\mathbf{x})\bullet \psi(\mathbf{y})$, hence $\psi$ is an isomorphism of the Jordan algebras, as required.

\end{proof}

\begin{remark}
It also is convenient to exhibit the Jordan multiplicative structure on $J(V,N,e)$  explicitly by polarizing the quadratic identity (\ref{multi01}), namely
\begin{equation}\label{multi0}
\mathbf{x}\bullet \mathbf{y}=(x_0y_0+Q(x,y), \,x_0y+y_0x+\frac{1}{3\sqrt{2}}\,h(x;y)),
\end{equation}
where $h$ is defined by (\ref{hessdef}).
\end{remark}

To finish the proof of Theorem~\ref{Theorem} it suffices  to verify the following

\begin{proposition}\label{cor:1}
Let $\alpha:\mathcal{J}\to \mathcal{E}$ and $\beta:\mathcal{E}\to \mathcal{J}$ are the correspondences defined in Proposition~\ref{cor:from} and Proposition~\ref{pro:th:1}, respectively. Then $\alpha  \beta=\mathrm{id}_\mathcal{E}$ and $\beta  \alpha=\mathrm{id}_\mathcal{J}$.

\end{proposition}

\begin{proof}
(i) We prove first that $\beta\alpha=id_{\mathcal{J}}$. Let  $J=J(V,N,e)$ and let $\alpha(J)=(e^\bot, \tau, u)$ be the cubic eiconal associated to $J$ by virtue of Proposition~\ref{cor:from}, where $u=\sqrt{2}N|_{e^\bot}$. By Proposition~\ref{pro:th:1}
$$
\beta(\alpha(J))=J(\widehat{V},\widehat{N},\mathbf{e}), \qquad \widehat{V}=\F\times \cbot,
$$
where $\widehat{N}$ is found  by (\ref{Norma}) as
\begin{equation*}
\begin{split}
\widehat{N}(\mathbf{x})&=x_0^3-\frac{3}{2} x_0\tau(x)+\frac{1}{\sqrt{2}}\,u(x)\\
&=x_0^3-\frac{3}{2} x_0\tau(x)+N(x), \qquad \mathbf{x}=(x_0,x)\in \widehat{V}.
\end{split}
\end{equation*}
We claim that $\phi(\mathbf{x})=x_0e+x:\beta(\alpha(J))\to J$ is an isomorphism of the Jordan algebras. Notice that $\phi$ is an $\F$-linear map and $\ker \phi=\{\mathbf{0}\}$. Thus, it suffices to verify that $\phi$ is compatible with the multiplicative structure  which is equivalent to showing that $\phi(\mathbf{x}^{ 2})=\phi(\mathbf{x})^{ 2}$. To this end we notice that by (\ref{multi01}) and the definition of $u$,
$
\mathbf{x}^{ 2}
=(x_0^2+Q(x), 2x_0x+ \frac{1}{3}\nabla N    (x)).
$
On the other hand, by the definition of $\phi$,
$$
\phi(\mathbf{x})^{ 2}=x_0^2e+2x_0x+x^{ 2},
$$
which readily yields
\begin{equation}\label{oo}
\begin{split}
\phi(\mathbf{x}^{ 2})-\phi(\mathbf{x})^{ 2}&= \frac{1}{3}\nabla N(x)-x^{ 2}+T(x)e,
\end{split}
\end{equation}
where the gradient is found by  (\ref{nablaN}) as follows: $\nabla N(x)=(\Nabla N(x))^\cbot=3(x^\#)^\cbot$. Furthermore, using $\Trace (x)=0$ and (\ref{sharpidentity}), we obtain
\begin{equation*}
\begin{split}
(x^\#)^\cbot&=(x^{ 2}-\Trace (x) x+S(x)e)^\cbot\equiv (x^{ 2})^\cbot\\
&=x^{ 2}-T(x^{ 2};e)e=x^{ 2}-T(x;x)e\\
&=x^{ 2}-T(x)e.
\end{split}
\end{equation*}
Thus (\ref{oo}) yields $\phi(\mathbf{x}^{ 2})-\phi(\mathbf{x})^{ 2}=0$,  which implies  the desired isomorphy of $\phi$.

(ii) In the converse direction, let us consider an eiconal triple $\epsilon=(W,Q,u)$ and associate to it the Jordan algebra $\beta(\epsilon)=J(V, N, \mathbf{e})$, where $V=\F\times W$, $N$ is defined by (\ref{Norma}), and $\mathbf{e}=(1,0)$. By (\ref{TT}), $\mathbf{e}^\bot=0\times W$, and  also by (\ref{Norma}) and (\ref{TT}), $N(\mathbf{x})=\frac{u(x)}{\sqrt{2}}$ and $\tau(\mathbf{x})=Q(x)$ for $\mathbf{x}=(0,x)\in \mathbf{e}^\bot$. Then Proposition~\ref{cor:from} yields:
$$
\alpha\beta(\epsilon)=(0\times W, Q, u)
$$
which is trivially equivalent to $\epsilon$. The proposition is proved.

\end{proof}

\bibliographystyle{plain}
\bibliography{main_references}

\begin{thebibliography}{10}

\bibitem{Duff1}
L.~Borsten, D.~Dahanayake, M.~J. Duff, H.~Ebrahim, and W.~Rubens.
\newblock Black holes, qubits and octonions.
\newblock {\em Phys. Rep.}, 471(3-4):113--219, 2009.

\bibitem{Cecil}
Th.~E. Cecil.
\newblock Isoparametric and {D}upin hypersurfaces.
\newblock {\em SIGMA Symmetry Integrability Geom. Methods Appl.}, 4:Paper 062,
  28, 2008.

\bibitem{Chaput02}
P.-E. Chaput.
\newblock Severi varieties.
\newblock {\em Math. Z.}, 240(2):451--459, 2002.

\bibitem{Chaput03}
P.-E. Chaput.
\newblock Scorza varieties and {J}ordan algebras.
\newblock {\em Indag. Math. (N.S.)}, 14(2):169--182, 2003.

\bibitem{Eting}
P.~Etingof, D.~Kazhdan, and A.~Polishchuk.
\newblock When is the {F}ourier transform of an elementary function elementary?
\newblock {\em Selecta Math. (N.S.)}, 8(1):27--66, 2002.

\bibitem{FKbook}
J.~Faraut and A.~Kor{\'a}nyi.
\newblock {\em Analysis on symmetric cones}.
\newblock Oxford Mathematical Monographs. The Clarendon Press Oxford University
  Press, New York, 1994.
\newblock Oxford Science Publications.

\bibitem{JacobOsaka}
N.~Jacobson.
\newblock Generic norm of an algebra.
\newblock {\em Osaka Math. J.}, 15:25--50, 1963.

\bibitem{JacobsonBook}
N.~Jacobson.
\newblock {\em Structure and representations of {J}ordan algebras}.
\newblock American Mathematical Society Colloquium Publications, Vol. XXXIX.
  American Mathematical Society, Providence, R.I., 1968.

\bibitem{Jordan33A}
P.~Jordan.
\newblock \"{U}ber verallgemeinerungsm\"{o}glichkeiten des formalismus der
  quantenmechanik.
\newblock {\em Nachr. Akad. Wiss. G\"ottingen. Math. Phys.}, 41:209--217, 1933.

\bibitem{JordanNeumann}
P.~Jordan, J.~von Neumann, and E.~Wigner.
\newblock On an algebraic generalization of the quantum mechanical formalism.
\newblock {\em Ann. of Math. (2)}, 35(1):29--64, 1934.

\bibitem{Kru}
S.~Krutelevich.
\newblock Jordan algebras, exceptional groups, and {B}hargava composition.
\newblock {\em J. Algebra}, 314(2):924--977, 2007.

\bibitem{McCrbook}
K.~McCrimmon.
\newblock {\em A taste of {J}ordan algebras}.
\newblock Universitext. Springer-Verlag, New York, 2004.

\bibitem{Mun1}
H.~F. M{\"u}nzner.
\newblock Isoparametrische {H}yperfl\"achen in {S}ph\"aren.
\newblock {\em Math. Ann.}, 251(1):57--71, 1980.

\bibitem{Mun2}
H.~F. M{\"u}nzner.
\newblock Isoparametrische {H}yperfl\"achen in {S}ph\"aren. {II}. \"{U}ber die
  {Z}erlegung der {S}ph\"are in {B}allb\"undel.
\newblock {\em Math. Ann.}, 256(2):215--232, 1981.

\bibitem{Nadir}
N.~Nadirashvili, V.G. Tkachev, and S.~Vl{\u{a}}du{\c{t}}.
\newblock A non-classical solution to a {H}essian equation from {C}artan
  isoparametric cubic.
\newblock {\em Adv. Math.}, 231(3-4):1589--1597, 2012.

\bibitem{Pirio1}
L.~Pirio and F.~Russo.
\newblock Extremal varieties 3-rationally connected by cubics, quadro-quadric
  cremona transformations and rank 3 jordan algebras.
\newblock {\em arXiv:1109.3573}, 2011.

\bibitem{Pirio2}
L.~Pirio and F.~Russo.
\newblock Quadro-quadric cremona transformations in low dimensionsv via the
  jc-correspondence.
\newblock {\em arXiv:1204.0428}, 2012.

\bibitem{Springer59}
T.~A. Springer.
\newblock On a class of {J}ordan algebras.
\newblock {\em Nederl. Akad. Wetensch. Proc. Ser. A 62 = Indag. Math.},
  21:254--264, 1959.

\bibitem{TkCartan}
V.~G. Tkachev.
\newblock A generalization of {C}artan's theorem on isoparametric cubics.
\newblock {\em Proc. Amer. Math. Soc.}, 138(8):2889--2895, 2010.

\end{thebibliography}

\end{document}